\documentclass[11pt]{scrartcl} 

\usepackage[latin1]{inputenc}
\usepackage{color,graphicx}
\usepackage{amsfonts,amssymb,amsthm,amsmath,booktabs}
\usepackage{placeins}

\ifx\pdftexversion\undefined
\usepackage{nohyperref}
\else 
\usepackage[colorlinks=true,linkcolor={blue},citecolor={blue},urlcolor={blue},
  pdfauthor={Stefan Dohr, Christian Kahle, Sergejs Rogovs, Piotr Swierczynski},
  pdfstartview={Fit}]{hyperref}
\fi

\usepackage[linesnumbered]{algorithm2e}

\usepackage{pgfplots}  
\usepackage{xcolor}
\usepackage{tikz}
\usetikzlibrary{plotmarks,backgrounds}
\usepgfplotslibrary{external}
 \tikzexternaldisable
 
\usepackage{todonotes}

\newcommand{\abssec}[1]{\noindent\normalsize {\bfseries #1\quad }\ignorespaces}
\newenvironment{acknowledgement}{\abssec{Acknowledgement.}}{\par\vspace{.1in}}

\theoremstyle{plain}
\newtheorem{theorem}{Theorem}
\newtheorem{corollary}[theorem]{Corollary}

\newtheorem{lemma}[theorem]{Lemma}

\newtheorem{assumption}[theorem]{Assumption}

\theoremstyle{definition}
\newtheorem{remark}[theorem]{Remark}

\numberwithin{equation}{section}

\def\N{\mathbb{N}}

\definecolor{darkgreen}{rgb}{0.0, 0.5, 0.3}

\newcommand{\norm}[1]{\left\lVert#1\right\rVert}

\DeclareMathOperator*{\argmin}{arg\,min}

\begin{document}

\title{A FEM for an optimal control problem subject to the fractional Laplace equation
\thanks{This work was supported by the International Research Training Group 1754,
 funded by the German Research Foundation (DFG), and the Austrian Science Fund
 (FWF).} 
 }

\author{Stefan Dohr~\footnotemark[2]\and
Christian Kahle~\footnotemark[3]\and
Sergejs Rogovs~\footnotemark[4]\and
Piotr Swierczynski~\footnotemark[3]
}

\date{\today}

\maketitle

\renewcommand{\thefootnote}{\fnsymbol{footnote}}
\footnotetext[2]{Institut f\"ur Angewandte Mathematik, Technische Universit\"at Graz, 8010 Graz, Austria, \texttt{stefan.dohr@tugraz.at}}
\footnotetext[3]{Center for Mathematical Sciences, Technische Universit\"at M\"unchen, 85748 Garching bei M\"unchen, Germany, 
\texttt{ \{christian.kahle,piotr.swierczynski\}@ma.tum.de}}
\footnotetext[4]{Institut f\"ur Mathematik und Computergest\"utzte Simulation, Universit\"at der Bundeswehr M\"unchen, 85577 Neubiberg, Germany, \texttt{sergejs.rogovs@unibw.de}}

\begin{abstract}
We study the numerical approximation of linear-quadratic 
optimal control problems subject to the fractional Laplace equation with its
spectral definition.
We compute an approximation of the state equation using a discretization of the
Balakrishnan formula that is based on a finite element discretization
in space and a sinc quadrature approximation of the additionally involved integral.
A tailored approach for the numerical solution of the resulting linear systems is proposed.
  
Concerning the discretization of the optimal 
control problem we consider two schemes. 
The first one is the variational approach,
where the control set is not discretized, 
and the second one is the fully discrete scheme where the 
control is discretized by piecewise constant functions. We derive finite element
error estimates for 
both methods and illustrate our results by numerical experiments.
\end{abstract}

\\
\\
%


\section{Introduction}
Let $\Omega \subset \mathbb{R}^{n}$ ($n \in \{2,3\}$) be a bounded and convex domain with boundary $\Gamma := \partial \Omega$ and $s \in (0,1)$. 
For $u_{d} : \Omega \rightarrow \mathbb{R}$ we
define the objective functional
\begin{equation}\label{eq:Functional}
J(u, z) := \frac{1}{2}\|u - u_d\|_{L^{2}(\Omega)}^{2} + \frac{\mu}{2} \|z\|_{L^{2}(\Omega)}^{2},
\end{equation}
where $\mu > 0$ denotes a regularization parameter. 
In this work we consider the optimal control problem of finding
\begin{equation}
\label{eq:min_J}
\argmin_z J(u,z),
\end{equation}
subject to the \textit{fractional state equation}
\begin{equation}
\label{eq:state_equation}
\left(-\Delta\right)^{s}u = z \textrm{ in } \Omega, \quad u=0 \textrm{ on } \Gamma,
\end{equation}
and the \textit{control constraints}
\begin{equation}
\label{eq:control_constraints}
a \leq z(x) \leq b \textrm{ a.e. in } \Omega,
\end{equation}
with constants $a, b \in \mathbb{R}$  satisfying $a \leq 0 \leq b$.
Here, we understand the operator $(-\Delta)^{s}$ in the sense of its spectral
definition, compare e.g.~\cite{Cabre, Capella, Nochetto}.
 
The main difficulty in studying this problem is the nonlocality of the fractional Laplace operator 
\cite{Cabre}. 
One way to overcome this issue is based on the  Cafarelli--Silvestre extension~\cite{Cafarelli} on unbounded domains
and its extension to bounded domains~\cite{Cabre, Capella, Stinga}.
 In this approach, an auxiliary problem in an extended domain 
 $\mathcal{C} := \Omega \times (0,\infty)$ is introduced and the solution of
  the state equation~\eqref{eq:state_equation} is then given as the Dirichlet trace on $\Omega \times \left\{0\right\}$ of the solution to the extended problem.
Exponential decay of the solution in the artificial dimension allows construction of different numerical methods, 
see e.g. \cite{Bonito, MeidPSV_2017_hpFE_fracDiff, Nochetto}.  
In these publications, the problem is discretized by introducing a tensor product mesh of the domain 
$\mathcal{C}_{\mathcal{Y}} = \Omega \times (0,Y)$, 
which is constructed by a conformal triangulation of $\Omega$ 
and a graded mesh in the artificial direction, 
see e.g. \cite[Section 5.1]{Nochetto}. 
A convergence rate of $h^{1 + s}$ (up to some logarithmic term) 
in the $L^2(\Omega)$-norm can be obtained~\cite{Antil, Nochetto2}, 
provided that $z \in \mathbb{H}^{1-s}(\Omega)$, where $h$ denotes the global mesh parameter.
However, numerical experiments show that this convergence rate is not optimal in a specific range of fractional powers~$s$.
The cost of solving the problem is related to the number of elements in $\mathcal{C}_{\mathcal{Y}}$,
and not only to the number of elements in $\Omega$, 
resulting in an increased computational complexity.
This issue is first overcome in \cite{MeidPSV_2017_hpFE_fracDiff} by 
exploiting $p$-finite elements in the extended direction.

An alternative approach for solving~\textup {(\ref {eq:state_equation})} uses the 
Balakrishnan representation formula \cite[IX. 11.]{Yosida_FuncAna}, namely for
$s \in (0,1)$ and $z \in \mathbb{H}^{-s}(\Omega)$
\begin{equation}
\label{eq:balakrishnan_formula}
(-\Delta)^{-s}z = \frac{\sin{(s\pi)}}{\pi} \int_{0}^{\infty} \nu^{-s} (\nu I - \Delta)^{-1}z d \nu.
\end{equation}
Numerical approximation of \textup {(\ref {eq:balakrishnan_formula})} is then based on a
suitable quadrature formula for \textup {(\ref {eq:balakrishnan_formula})} with respect to
$\nu$ and a discretization of the operator $\nu I-\Delta$ using the finite element method,
see \cite{Bonito}.

While the numerical analysis of the optimal control problem
\eqref{eq:min_J}--\eqref{eq:control_constraints} using an equivalent formulation with the Cafarelli--Silvestre extension is well established \cite{Antil}, the numerical analysis using the Balakrishnan formula is still open.

In this article we propose and analyze two discrete schemes for the approximation of the solution to the
 optimal control problem \eqref{eq:min_J}--\eqref{eq:control_constraints} using
 the Balakrishnan representation of the solution $u$ of the state equation
 \eqref{eq:state_equation}.
 Both schemes rely on a finite element discretization of the operator $\nu I - \Delta$ in \eqref{eq:balakrishnan_formula} 
 and a sinc quadrature approximation \cite{2018_BonitoLeiPasciak} of the integral in \eqref{eq:balakrishnan_formula}. 
 The first method is the variational discretization approach \cite{Hinze}, 
 where the set of controls is not discretized a priori. However, it inherits its approximation properties from the approximation of the adjoint state. 
 The second one uses a fully discrete setting, where the set of controls is discretized by piecewise constant functions \cite{Arada2002, Casas2005, Roesch2006}. 
 We derive $L^2(\Omega)$-error estimates for the state and control for both types of the FE discretization of the optimal control problem.

Regarding the variational approach for the discretization of the optimal control problem \eqref{eq:min_J}--\eqref{eq:control_constraints} 
we show an optimal convergence rate of $h^{\min{(2, 3/2+2s - \varepsilon)}}$ for the control and the state in the $L^{2}(\Omega)$-norm, 
whereas using the extension approach \cite{Antil} yields a convergence rate of $h^{1+s}$ (up to some logarithmic term). 
In the case of the fully discrete scheme we show the expected linear convergence for the control in the $L^{2}(\Omega)$-norm 
and for the state in the $\mathbb{H}^{s}(\Omega)$-norm. 
Numerically we also consider the post-processing approach \cite{MeyerRoesch2004}
for the optimal control and measure again the same rate of  $h^{\min{(2, 3/2+2s - \varepsilon)}}$ for the 
post-processed optimal control.
Similar results are shown for the extension approach in \cite{Antil}.
While the convergence rate for the optimal control is optimal, as confirmed by numerical experiments, 
there is still a gap between the theoretical and practical rates for the state, 
which will be addressed in future work. 

The outline of this paper is as follows.
 In Section~\ref{sec:fractional_optimal_control} we review existence and uniqueness results for the 
 fractional optimal control problem based on \cite{Antil} as well as regularity properties of the optimal control problem. 
 The numerical analysis of the mentioned discretization methods is conducted in Section~\ref{sec:error_estimates}, 
 starting with the derivation of the error estimates for the discretization of the state equation \eqref{eq:state_equation} 
 using the Balakrishnan formula. 
 In Section~\ref{subsec:semidiscrete_scheme} we study the convergence properties of the optimal control 
 and state using the semidiscrete approach, 
 while Section~\ref{subsec:full_discrete_scheme} is devoted to the numerical analysis of the fully discrete scheme. 
 In Section~\ref{sec:Implementation} we introduce a solver for the finite element approximation of the problem. 
 Numerical results validating the theoretical convergence results for the proposed discretization techniques are presented in Section~\ref{sec:Numerics}.

\section{Existence and regularity of optimal controls}
\label{sec:fractional_optimal_control}
In this section we review existence and uniqueness as well as the regularity results for the 
optimal control problem \eqref{eq:min_J}--\eqref{eq:control_constraints} based on \cite[Sec. 3]{Antil}.
We start this section with a brief introduction of the  spectral definition of the fractional operator
$(-\Delta)^s$ following~\cite{Cabre, Capella}.

The eigenfunctions $\left\{\varphi_{k}\right\}_{k\in \N}$ with eigenvalues $\left\{\lambda_{k}\right\}_{k \in \N}$ of the Laplace operator, i.e.,
\begin{equation*}
-\Delta \varphi_{k} = \lambda_{k} \varphi_{k} \textrm{ in } \Omega, \quad \varphi_{k} = 0 \textrm{ on } \Gamma, \quad k \in \N
\end{equation*}
form an orthonormal basis of $L^{2}(\Omega)$. The spectral fractional Laplace operator for ${w \in C_{0}^{\infty}(\Omega)}$ is then defined as
\begin{equation*}
(-\Delta)^{s}w:= \sum_{k = 1}^{\infty} \lambda_{k}^{s} w_{k} \varphi_{k}, \quad w_{k} := \int_{\Omega} w \varphi_{k} dx, \quad k \in \N.
\end{equation*}
This definition can be extended by density to the space $\mathbb{H}^{s}(\Omega)$ \cite{Antil, Nochetto} defined as
\begin{equation}
\label{eq:fractional_sobolev_spaces}
\mathbb{H}^{s}(\Omega) := \left\{w = \sum_{k=1}^{\infty} w_{k} \varphi_{k} : \sum_{k=1}^{\infty} \lambda_{k}^{s}w_{k}^{2} < \infty \right\} = 
\begin{cases}
H^{s}(\Omega) \equiv H_{0}^{s}(\Omega) & \textrm{if } s \in (0,\frac{1}{2}),\\
H_{00}^{1/2}(\Omega) & \textrm{if } s = \frac{1}{2}, \\
H_{0}^{s}(\Omega) & \textrm{if } s \in (\frac{1}{2}, 1).
\end{cases}
\end{equation}
The characterization of the fractional Sobolev spaces on the right hand side in \eqref{eq:fractional_sobolev_spaces} can be found, e.g., in \cite{mclean2000strongly}. For $s \in [1,2]$ we set $\mathbb{H}^{s}(\Omega) := H^{s}(\Omega) \cap H^{1}_{0}(\Omega)$, whereas $\mathbb{H}^{0}(\Omega) := L^{2}(\Omega)$. The dual space  of $\mathbb{H}^{s}(\Omega)$ we denote by $\mathbb{H}^{-s}(\Omega)$.
We stress that this definition of $(-\Delta)^s$ inherently assumes homogeneous Dirchlet boundary data (in a suitable sense). 
For a generalization to inhomogeneous Dirichlet boundary data we refer to \cite{Antil_Pfefferer_Rogovs} and the references therein.

\bigskip

Let $u_{d} \in L^{2}(\Omega)$ and $a, b \in \mathbb{R}$ with $a \leq 0 \leq b$ be given.
We define the set of admissible controls $Z_{ad}$ by
\begin{equation*}
Z_{ad} := \left\{w \in L^{2}(\Omega): a \leq w(x) \leq b \textrm{ a.e. in } \Omega \right\}.
\end{equation*}
Let  ${S: \mathbb{H}^{-s}(\Omega) \rightarrow \mathbb{H}^{s}(\Omega)}$ denote the control to state operator defined as $Sz := u$, 
where $u \in \mathbb{H}^{s}(\Omega)$ is the unique solution of the state equation \eqref{eq:state_equation}. 
It holds, that for any $z \in \mathbb{H}^{-s}(\Omega)$, 
the boundary value problem \eqref{eq:state_equation} has a unique solution $u \in \mathbb{H}^{s}(\Omega)$, see e.g. \cite{MeidPSV_2017_hpFE_fracDiff}.
We may also consider the operator $S$ acting on $L^{2}(\Omega)$ with range in $L^{2}(\Omega)$. Note also that $S$ is self-adjoint, since the operator $(-\Delta)^{s}$ is self-adjoint. The adjoint state $p \in \mathbb{H}^{s}(\Omega)$ for $z \in \mathbb{H}^{-s}(\Omega)$ is then given by $p = S(Sz - u_{d})$. 
In \cite[Section 3.1]{Antil} the existence and uniqueness of a solution to the optimal control problem \eqref{eq:min_J}--\eqref{eq:control_constraints} is shown. Let us recall the main result from that reference.

\begin{theorem}[existence, uniqueness, and optimality conditions, {\cite[Section 3.1]{Antil}}]
The fractional optimal control problem \eqref{eq:min_J}--\eqref{eq:control_constraints} has a unique optimal solution 
$(\bar{u}, \bar{z}) \in \mathbb{H}^{s}(\Omega) \times Z_{ad}$. 
These fulfill the  necessary and sufficient optimality conditions
\begin{align}
  \bar{u} &= S \bar{z} \in \mathbb{H}^{s}(\Omega),\\
  \bar{p} &= S(\bar{u}-u_{d}) \in \mathbb{H}^{s}(\Omega),\\
  \bar{z} \in Z_{ad}, &\quad (\mu \bar{z} + \bar{p}, z - \bar{z})_{L^{2}(\Omega)} \geq 0 \quad \textrm{for all } z \in Z_{ad}.
  \label{eq:variational_inequality}
\end{align}
\end{theorem}

For $\mu > 0$ and $\bar{p} = S(\bar{u}-u_{d})$ the variational inequality \eqref{eq:variational_inequality} is equivalent to the projection formula \cite{Fredi}
\begin{equation*}
\bar{z}(x) = \textrm{proj}_{[a, b]} \left(-\frac{1}{\mu} \bar{p}(x)\right)
\end{equation*}
where $\textrm{proj}_{[a, b]}(v) := \min{\left\{b, \max{\left\{a,v\right\}}\right\}}$.
Since we assume that $\Omega$ is a convex domain and that $a \leq 0 \leq b$ we can prove the following regularity results for the control.
\begin{lemma}[$H^{1}$-regularity of the optimal control, {\cite[Lemma 3.5]{Antil}}]
Let $\bar{z} \in Z_{ad}$ be the optimal control and $u_{d} \in \mathbb{H}^{1-s}(\Omega)$. 
Then $\bar{z} \in H_0^{1}(\Omega)$.
\end{lemma}
\begin{proof}
The proof is based on bootstrapping. 
We only comment on the case $s \in (0,\frac{1}{4})$. In this case,
an intermediate regularity result is $\bar{z} \in \mathbb{H}^{s}(\Omega)$. 
As $\bar z = \textrm{proj}_{[a,b]}\left( -\frac{1}{\mu}\bar p\right)$ this,
in turn, requires $a \leq 0 \leq b$.
\end{proof}
\begin{lemma}
 \label{lem:regularity_control}
Let $\bar{z} \in Z_{ad}$ be the solution of the optimal control problem \textup {(\ref {eq:min_J})}--\textup {(\ref {eq:control_constraints})} 
with $u_{d} \in \mathbb{H}^{3/2}(\Omega)$. 
Then $\bar{z} \in \mathbb{H}^{3/2-\varepsilon}(\Omega)$, where $\varepsilon$ is a positive, arbitrary small number.
\end{lemma}
\begin{proof}
The proof follows from the standard bootstraping argument.
\end{proof}

\section{A priori error estimates}
\label{sec:error_estimates}

In this section, we analyse two finite element approximations of 
the fractional optimal control problem \eqref{eq:min_J}--\eqref{eq:control_constraints}. 
First, we investigate the variational approach \cite{Hinze}, 
where the control set is not discretized, and then move to a fully discrete scheme. 
Both techniques are based on a finite element discretization of the state 
equation \eqref{eq:state_equation} using the Balakrishnan formula \eqref{eq:balakrishnan_formula}. 
In the following subsection, we review the resulting FE error estimates, based on \cite[Sec. 4]{Bonito}.

\begin{assumption} 
Throughout this and the following sections we assume that $\Omega$ is a polygonal or polyhedral domain and that $u_{d} \in \mathbb{H}^{3/2}(\Omega)$, hence the regularity result from Lemma \ref{lem:regularity_control} holds.
\end{assumption}

\subsection{A finite element method for the state equation}
\label{subsec:fem_state_equation}

Let $\mathbb{U}(\mathcal{T}_{h})$ be the space of piecewise linear and globally continuous functions vanishing on the boundary $\partial \Omega$, defined with respect to a conforming quasi-uniform triangulation $\mathcal{T}_{h}$ of the domain $\Omega$. 
A FE approximation of problem \eqref{eq:state_equation} for $z \in L^{2}(\Omega)$ is given by
\begin{equation}
\label{eq:dunford_taylor_FE}
u_{h} = \frac{\sin{(s \pi)}}{\pi} \int_{0}^{\infty} \nu^{-s}(\nu I - \Delta_{h})^{-1} z\,d\nu
= \frac{\sin{(s \pi)}}{\pi}\int_{-\infty}^{\infty} e^{(1-s)t}\left(e^tI-\Delta_h\right)^{-1}z\,dt,
\end{equation}
where $\Delta_{h}$ denotes the discrete Laplace operator.

For $k > 0$ we define the numbers
$N_+ := \Bigg\lceil \frac{\pi^2}{4 s k^2} \Bigg\rceil$ and $N_- := \Bigg\lceil \frac{\pi^2}{4 (1-s) k^2} \Bigg\rceil$.
The sinc quadrature approximation of $u_{h}$ is then given by
\begin{equation}
\label{eq:sinc_quadrature}
u_{h}^{k} := \frac{\sin{(s \pi)}}{\pi} k \sum_{l = -N_{-}}^{N_{+}} e^{(1-s)kl}(e^{kl}I - \Delta_{h})^{-1} z.
\end{equation}

Practical aspects of the numerical implementation of this method are discussed in Section\nobreakspace \ref {sec:Implementation}.

In our problem set-up the following error estimates hold.

\begin{theorem}[finite element approximation,{\cite[Theorem 4.2]{Bonito}}]
\label{thm:fe_error_dunford_talyor} 
Given $r  \in [0,1]$ with $r \leq 2s$, set $\gamma := \max{(r+2\alpha_{\star} - 2s, 0)}$ 
and $\alpha_{\star} := \frac{1}{2}(\alpha + \min{(1-r, \alpha)})$ with ${\alpha \in (0,1]}$. 
If $z \in \mathbb{H}^{\delta}(\Omega)$ for $\delta \geq \gamma$, then 
\begin{equation*}
\norm{u - u_{h}}_{\mathbb{H}^{r}(\Omega)} \leq C_{h}\, h^{2\alpha_{\star}} \norm{z}_{\mathbb{H}^{\delta}(\Omega)}
\end{equation*}
where $C_{h} \leq c \log{(2/h)}$ if $\delta = \gamma$ and $r+2\alpha_{\star} \geq 2s$, and $C_{h} \leq c$ otherwise. 
\end{theorem}
Note, that we get a convergence rate of $h^{2 - r}$ if we set $\alpha = 1$ and if $z$ is regular enough. However, in order to obtain the convergence rates depending on $s \in (0,1)$ and on the regularity of $z$, we have to choose $\alpha$ in Theorem \ref{thm:fe_error_dunford_talyor} appropriately.
For $r=0$ and $r=s$ respectively in Theorem \ref{thm:fe_error_dunford_talyor}, we conclude the following error estimates.
\begin{corollary}
For $z \in \mathbb{H}^{\delta + \varepsilon'}(\Omega)$ with $\varepsilon' > 0$ arbitrary small and $\delta \geq - \varepsilon'$ there holds
\begin{equation}
\label{eq:fe_error_l2_hs}
\begin{aligned}
\norm{u - u_{h}}_{L^{2}(\Omega)} &\leq c \, h^{\min{(2, \delta + 2s)}} \norm{z}_{\mathbb{H}^{\delta + \varepsilon'}(\Omega)},\\
\norm{u - u_{h}}_{\mathbb{H}^{s}(\Omega)} &\leq c \, h^{\min{(2-s, \delta + s)}} \norm{z}_{\mathbb{H}^{\delta + \varepsilon'}(\Omega)}.
\end{aligned}
\end{equation}
\end{corollary}

The quadrature formula \eqref{eq:sinc_quadrature} possesses the following approximation property.
\begin{theorem}[sinc quadrature approximation, {\cite[Theorem 4.3]{2018_BonitoLeiPasciak}}]
For $r \in [0,1]$ and $z \in \mathbb{H}^{r}(\Omega)$ there holds
\begin{equation*}
\norm{u_{h} - u_{h}^{k}}_{\mathbb{H}^{r}(\Omega)} 
\leq c \, e^{-\pi^{2}/(2k)} \norm{z}_{\mathbb{H}^{\max(0,r-2s+\epsilon)}(\Omega)}
\leq c \, e^{-\pi^{2}/(2k)} \norm{z}_{\mathbb{H}^{r}(\Omega)}
\end{equation*}
\end{theorem}

Hence, if we choose $k$ appropriately, we can balance the sinc quadrature and the finite element errors.
\begin{lemma}
\label{lem:fe_approximation_error}
Assume that the number of integration points in the
 sinc quadrature~\textup {(\ref {eq:sinc_quadrature})} is balanced with the FE errors
 \textup {(\ref {eq:fe_error_l2_hs})}, i.e., $k \in \mathcal O(\big|\ln{h}\big|^{-1})$. For
 $z \in \mathbb{H}^{\delta+\varepsilon'}(\Omega)$ with $\varepsilon' > 0$ and $\delta \geq - \varepsilon'$ we obtain
\begin{equation}
\label{eq:fe_approximation_error}
\begin{aligned}
\norm{u - u_{h}^{k}}_{L^{2}(\Omega)} &\leq c \, h^{\min{(2, \delta+2s)}} \norm{z}_{\mathbb{H}^{\delta+\varepsilon'}(\Omega)},\\
\norm{u - u_{h}^{k}}_{\mathbb{H}^{s}(\Omega)} &\leq c \, h^{\min{(2-s, \delta+s)}} \norm{z}_{\mathbb{H}^{\delta+\varepsilon'}(\Omega)}.
\end{aligned}
\end{equation}
\end{lemma}

Given the regularity results of ~Lemma\nobreakspace \ref {lem:regularity_control} for $u_{d} \in \mathbb{H}^{3/2}(\Omega)$ we conclude the following error estimates.
\begin{corollary}
\label{cor:fe_error_estimates_regular}
For $z \in \mathbb{H}^{3/2-\varepsilon}(\Omega)$ there holds
\begin{equation*}
\norm{u - u_{h}^{k}}_{L^{2}(\Omega)} \leq c \, h^{\min{(2, 3/2+2s-\varepsilon')}} \norm{z}_{\mathbb{H}^{3/2-\varepsilon}(\Omega)}
\end{equation*} 
and
\begin{equation*}
\norm{u - u_{h}^{k}}_{\mathbb{H}^{s}(\Omega)} \leq c \, h^{\min{(2-s, 3/2 + s - \varepsilon')}} \norm{z}_{\mathbb{H}^{3/2-\varepsilon}(\Omega)}
\end{equation*}
with $\varepsilon' > 0$ and $\varepsilon > 0$ arbitrary small and $\varepsilon < \varepsilon'$. Note, that the approximation $u^k_{h}$ converges quadratically in the
$L^{2}(\Omega)$-norm, provided that $s > \frac{1}{4}$.
\end{corollary}

In the following we drop the superscript $k$ and write $z_h$, $u_h$, and $p_h$ for the discrete approximations of $z$, $u$ and $p$.


\subsection{Variational discretization}
\label{subsec:semidiscrete_scheme}

We define the variational discretization of the optimal control problem
\eqref{eq:min_J}--\eqref{eq:control_constraints}
as finding $\bar{z}_h \in Z_{\text{ad}}$ such that
\begin{align}\label{eq:var_dis}
\bar{z}_h := \argmin_{z \in Z_{\text{ad}}}J_h(z) = \argmin_{z \in Z_{\text{ad}}}\frac{1}{2} \lVert S_h z -u_d \rVert_{L^2{(\Omega})}^{2} + \frac{\mu}{2}\lVert z\rVert_{L^2{(\Omega})}^{2}.
\end{align}

Similarly, as in the continuous setting, the discrete optimal control problem \eqref{eq:var_dis} has a unique solution $\bar{z}_h \in Z_{\text{ad}}$.
We denote by $\bar{u}_h:= S_h \bar{z}_h$  the optimal discrete state and by 
$\bar{p}_h := S_h(S_h \bar{z}_h-u_d)$ the optimal discrete adjoint state. In this case the variational inequality reads as
\begin{equation}\label{eq:var_ineq_var}
 (\bar{p}_h + \mu \bar{z}_h,z-\bar{z}_h)_{L^2(\Omega)} \geq 0\quad \forall z \in Z_{ad},
\end{equation}
which implies 
 \begin{equation}
 \label{eq:disc_proj}
  \bar{z}_h = \textrm{proj}_{[a,b]}\left( -\frac{1}{\mu}\bar{p}_h \right).
 \end{equation}
 Here and in the following, we denote by $S_h$ the discrete, self-adjoint solution operator defined by \textup {(\ref {eq:sinc_quadrature})}.
 
\begin{lemma}\label{lemma:stab}
The following stability estimates hold
\begin{align}
\lVert Sv \rVert_{L^2(\Omega)} &\leq c \lVert v \rVert_{L^2(\Omega)}, \quad \lVert S_h v \rVert_{L^2(\Omega)} \leq c \lVert v \rVert_{L^2(\Omega)}, \quad \lVert S_h v \rVert_{\mathbb{H}^s(\Omega)} \leq c \lVert v \rVert_{L^2(\Omega)}.
  \label{eq:Sv}
\end{align}
\end{lemma}
\begin{proof}
The first estimate follows from a trivial embedding and the $2s$-shift of the fractional Laplace operator. 
\[
\lVert Sv \rVert_{L^2(\Omega)} \leq c \lVert Sv \rVert_{\mathbb{H}^{2s}(\Omega)} \leq c \lVert v \rVert_{L^2(\Omega)} .
\]
To prove the second estimate we introduce the intermediate function $Sv$ to obtain
\[
\norm{S_{h}v}_{L^{2}(\Omega)} \leq \lVert (S_h - S) v \rVert_{L^2(\Omega)} + \lVert Sv \rVert_{L^2(\Omega)},
\]
and apply  the a priori error estimate \eqref{eq:fe_approximation_error} 
with $\delta = -\varepsilon'$ 
as well as the first estimate in \eqref{eq:Sv}. 
The proof of the third estimate follows the same path, using the stability of the operator $S: L^{2}(\Omega) \rightarrow \mathbb{H}^{s}(\Omega)$.
\end{proof}
 
\begin{theorem}
Let the pairs $(\bar{u}(\bar{z}),\bar{z})$ and $(\bar{u}_h(\bar{z}_h),\bar{z}_h)$ be the solutions to problems \eqref{eq:min_J} and \eqref{eq:var_dis}, respectively. Then the estimates
\begin{align}\label{eq:var_a_priori_z}
\lVert \bar{z} - \bar{z}_h \rVert_{L^2(\Omega)} &\leq c \, h^{\min(2,3/2+2s -\varepsilon')} \left(\lVert \bar{z} \rVert_{\mathbb{H}^{3/2-\varepsilon}(\Omega)} + \lVert u_d \rVert_{\mathbb{H}^{3/2-\varepsilon}(\Omega)} \right),\\\label{eq:var_a_priori_u}
\lVert \bar{u} - \bar{u}_h \rVert_{L^2(\Omega)} &\leq c \, h^{\min(2,3/2+2s -\varepsilon')} \left(\lVert \bar{z} \rVert_{\mathbb{H}^{3/2-\varepsilon}(\Omega)} + \lVert u_d \rVert_{\mathbb{H}^{3/2-\varepsilon}(\Omega)} \right),\\\label{eq:var_a_priori_u_s}
\lVert \bar{u} - \bar{u}_h \rVert_{\mathbb{H}^s(\Omega)} &\leq c \, h^{\min(2-s,3/2+s -\varepsilon')} \left(\lVert \bar{z} \rVert_{\mathbb{H}^{3/2-\varepsilon}(\Omega)} + \lVert u_d \rVert_{\mathbb{H}^{3/2-\varepsilon}(\Omega)} \right)
\end{align}
hold, provided $\varepsilon < \varepsilon'$.
\end{theorem}
\begin{proof}
We begin by showing the first estimate.
The proof is similar to the proof of \cite[Theorem 5.10]{Antil} based on ideas introduced in \cite{Hinze}.
Testing variational inequalities \eqref{eq:variational_inequality} and \eqref{eq:var_ineq_var} with $\bar{z}_h \in Z_{\text{ad}}$ and $\bar{z} \in Z_{\text{ad}}$, respectively, and adding both expressions, we arrive at
\begin{align}
\mu \lVert \bar{z}-\bar{z}_h \rVert_{L^2(\Omega)}^2& \leq ( \bar{p}-\bar{p}_h, \bar{z}_h-\bar{z})\nonumber\\
 &\leq ((S-S_h)S\bar{z}, \bar{z}_h-\bar{z})+ (S_{h}(S-S_h)\bar{z}, \bar{z}_h-\bar{z})\nonumber\\
 &\quad +((S_h-S) u_d, \bar{z}_h-\bar{z}) + (S_h^2(\bar{z}-\bar{z}_h), \bar{z}_h-\bar{z}).\label{eq:4_terms_var}
\end{align}
The first two terms can be estimated using the Cauchy--Schwarz inequality, Lemma \ref{lemma:stab} and the a priori estimate \eqref{eq:fe_approximation_error}
\begin{align}
((S-S_h)S\bar{z}, \bar{z}_h-\bar{z}) &\leq c \, h^{\min(2,3/2+2s -\varepsilon')}
\lVert \bar{z} \rVert_{\mathbb{H}^{3/2-\varepsilon}(\Omega)}\lVert
\bar{z}-\bar{z}_h \rVert_{L^2(\Omega)},\label{eq:var_proof_1}\\
(S_h(S-S_h)\bar{z}, \bar{z}_h-\bar{z}) &\leq c \, h^{\min(2,3/2+2s -\varepsilon')} \lVert \bar{z} \rVert_{\mathbb{H}^{3/2-\varepsilon}(\Omega)}\lVert \bar{z}-\bar{z}_h \rVert_{L^2(\Omega)}.
\end{align}
The estimate of the third term follows from the Cauchy--Schwarz inequality and the estimate \eqref{eq:fe_approximation_error}
\begin{equation}\label{eq:var_proof_3}
((S_h-S) u_d, \bar{z}_h-\bar{z}) \leq c \, h^{\min(2,3/2+2s -\varepsilon')} \lVert u_d \rVert_{\mathbb{H}^{3/2-\varepsilon}(\Omega)} \lVert \bar{z}-\bar{z}_h \rVert_{L^2(\Omega)},
\end{equation}
 and the last term is non-positive, since $S_h$ is self-adjoint and therefore
\begin{equation}\label{eq:var_proof_4}
(S_h^2(\bar{z}-\bar{z}_h), \bar{z}_h-\bar{z}) \leq -\lVert S_h(\bar{z}-\bar{z}_h) \rVert_{L^2(\Omega)}^2 \leq 0.
\end{equation}
The desired estimate follows from estimates \eqref{eq:var_proof_1}--\eqref{eq:var_proof_4}.

Application of Lemma\nobreakspace \ref {lem:fe_approximation_error}, Lemma\nobreakspace \ref {lemma:stab} and \textup {(\ref {eq:var_a_priori_z})} leads to 
\begin{align*}
\lVert \bar{u} - \bar{u}_h \rVert_{L^2(\Omega)} &\leq \lVert S \bar{z} -S_h \bar{z} \rVert_{L^2(\Omega)} + \lVert S_h \bar{z} -S_h \bar{z}_h \rVert_{L^2(\Omega)} \\
&\leq  c \, h^{\min(2,3/2+2s -\varepsilon')} \left(\lVert \bar{z} \rVert_{\mathbb{H}^{3/2-\varepsilon}(\Omega)} + \lVert u_d \rVert_{\mathbb{H}^{3/2-\varepsilon}(\Omega)} \right),
\end{align*}
and this proves \textup {(\ref {eq:var_a_priori_u})}. The proof of
\textup {(\ref {eq:var_a_priori_u_s})} follows the same path.
\end{proof}


\subsection{A fully discrete scheme}
\label{subsec:full_discrete_scheme}

In this section we consider a fully discrete   
scheme for the optimal control problem \eqref{eq:min_J}--\eqref{eq:control_constraints}. 
We discretize the set of admissible controls with piecewise constant functions
\begin{align*}
 Z_h &:= \{ z_h \in L^{\infty}(\Omega) : z_h \vert_T \in \mathcal{P}_0 \text{ for all } T \in \mathcal{T}_h \}, 
 \quad \text{and }Z_h^{\text{ad}} := Z_h \cap Z_{\text{ad}} .
\end{align*}
The discretized optimal control problem reads as: find $\bar z_h \in Z_h^{\text{ad}}$ such that
\begin{align}\label{eq:opt_cont_pb_full}
\bar{z}_h &= \argmin_{z_h \in Z^{\text{ad}}_h}J_h (z_h) =  \argmin_{z_h \in Z^{\text{ad}}_h} \frac{1}{2} \lVert S_h z_h - u_d \rVert^2_{L^2(\Omega)} + \frac{\mu}{2} \lVert z_h \rVert^2_{L^2(\Omega)}.
\end{align}
Using the same argumentation as in the continuous case, 
it can be shown that the optimal control problem \eqref{eq:opt_cont_pb_full} 
has a unique solution $\bar{z}_h \in Z_{h}^{\text{ad}}$.
 Let $\bar{u}_h = S_h\bar{z}_h$ and $\bar{p}_h = S_h(S_h \bar{z}_h - u_d)$ 
 be the optimal discrete state and optimal discrete adjoint state,
 respectively, associated with $\bar{z}_h $. 
 Then the discrete optimality condition reads as 
 \begin{equation}\label{eq:full_ineq_var}
  (\bar{p}_h + \mu \bar{z}_h, z_h - \bar{z}_h)_{L^2(\Omega)} \geq 0\quad \forall z_h \in Z_h^{\text{ad}}.
 \end{equation}
Before we state the main result of this section, 
we define the $L^2(\Omega)$-projection operator $Q_h: L^2(\Omega) \rightarrow Z_h$ by
\[
\int_{\Omega} (z - Q_h z) v_h = 0 \quad \forall v_h \in Z_h,
\]
which has the following properties
\begin{enumerate}
\item[(L1)] $\lVert Q_h v \rVert_{L^2(\Omega)} \leq c\lVert  v \rVert_{L^2(\Omega)} \quad \forall v \in L^2(\Omega)$,
\item[(L2)] $\lVert v - Q_h v \rVert_{L^2(\Omega)} \leq c \, h \norm{v}_{\mathbb{H}^{1}(\Omega)} \quad \forall v \in {H}^1(\Omega)$.
\end{enumerate}

\begin{theorem}
\label{thm:FD:rates}
Let the pairs $(\bar{u}(\bar{z}),\bar{z})$ and $(\bar{u}_h(\bar{z}_h),\bar{z}_h)$ 
be the solutions to problems \eqref{eq:min_J} and \eqref{eq:opt_cont_pb_full}, respectively. Then the estimates
\begin{align}\label{eq:full_a_priori_z}
\lVert \bar{z} - \bar{z}_h \rVert_{L^2(\Omega)} &\leq c \, h \left(\norm{\bar{z}}_{\mathbb{H}^{1}(\Omega)} 
+ \lVert u_d \rVert_{\mathbb{H}^{\max{(0, 1-2s+\varepsilon)}}(\Omega)} \right),\\
\label{eq:full_a_priori_u}
\lVert \bar{u} - \bar{u}_{h} \rVert_{\mathbb{H}^{s}(\Omega)} &\leq c \, h \left(\lVert \bar{z} \rVert_{\mathbb{H}^{1}(\Omega)} 
+ \lVert u_d \rVert_{\mathbb{H}^{\max{(0, 1-2s+\varepsilon)}}(\Omega)} \right),\\
\lVert \bar{u} - \bar{u}_{h} \rVert_{L^{2}(\Omega)} &\leq c \, h \left(\lVert \bar{z} \rVert_{\mathbb{H}^{1}(\Omega)} 
+ \lVert u_d \rVert_{\mathbb{H}^{\max{(0, 1-2s+\varepsilon)}}(\Omega)} \right)
\end{align}
hold.
\end{theorem}
\begin{proof}
The proof is similar to the proof of \cite[Theorem 5.16]{Antil}. 
First, we use $z = \bar{z}_h \in Z_{\textrm{ad}}$ in the continous optimality condition \eqref{eq:variational_inequality} 
to get
\[
(\bar{p} + \mu \bar{z},\bar{z}_h - \bar{z}) \geq 0. 
\]
Second, using $z_h = Q_h \bar{z} \in Z_h^{\text{ad}}$ in the discrete optimality condition
\eqref{eq:full_ineq_var}
and introducing $\bar{z}$, we arrive at
\[
(\bar{p}_h + \mu \bar{z}_h , Q_h \bar{z}- \bar{z}) + (\bar{p}_h + \mu \bar{z}_h , \bar{z}- \bar{z}_h)  \geq 0.
\]
Consequently, adding the previous two inequalities together we get
\[
(\bar{p} -\bar{p}_h + \mu (\bar{z} - \bar{z}_h) ,\bar{z}_h - \bar{z}) + (\bar{p}_h + \mu \bar{z}_h , Q_h \bar{z}- \bar{z})  \geq 0.
\]
Hence, we can conclude
\begin{equation}\label{eq:full_2_terms}
\mu \lVert \bar{z} - \bar{z}_h \rVert_{L^2(\Omega)}^2 \leq  (\bar{p} -\bar{p}_h  ,\bar{z}_h - \bar{z}) + (\bar{p}_h + \mu \bar{z}_h , Q_h \bar{z}- \bar{z}).
\end{equation}
The estimate for the first term on the right hand side of \eqref{eq:full_2_terms} 
follows from the estimate for \eqref{eq:4_terms_var} 
with an appropriate application of estimate \eqref{eq:fe_approximation_error}
\begin{equation}\label{eq:full_2_1_terms}
(\bar{p} -\bar{p}_h  ,\bar{z}_h - \bar{z}) \leq c \, h \left( \lVert \bar{z} \rVert_{\mathbb{H}^{1}(\Omega)} + \lVert u_d \rVert_{\mathbb{H}^{\max{(0, 1-2s+\varepsilon)}}(\Omega)} \right) \lVert \bar{z}-\bar{z}_h \rVert_{L^2(\Omega)}.
\end{equation}
To estimate the second term we add and substract $\bar{p}$ and $\mu \bar{z}$ and get 
\begin{equation}\label{eq:full_3_terms}
(\bar{p}_h + \mu \bar{z}_h , Q_h \bar{z}- \bar{z})  
= (\bar{p} + \mu \bar{z} , Q_h \bar{z}- \bar{z}) + \mu (\bar{z}_h - \bar{z} , Q_h \bar{z}- \bar{z}) + (\bar{p}_h - \bar{p} , Q_h \bar{z}- \bar{z}).
\end{equation}
To estimate the first term on the right hand side of \eqref{eq:full_3_terms} we use the definition of the operator $Q_h$ and obtain
\[
(\bar{p} + \mu \bar{z} , Q_h \bar{z}- \bar{z})  = (\bar{p} + \mu \bar{z} - Q_h(\bar{p} + \mu \bar{z}) , Q_h \bar{z}- \bar{z}) \leq c h^2 \lVert \bar{p} + \mu \bar{z} \rVert_{\mathbb{H}^{1}(\Omega)} \lVert  \bar{z} \rVert_{\mathbb{H}^{1}(\Omega)},
\]
where the last inequality follows from property (L2) of the $L^2$-projection. 
The application of the Cauchy-Schwarz inequality yields the desired estimate of the second term 
\begin{equation}\label{eq:full_3_1_terms}
\mu (\bar{z}_h - \bar{z} , Q_h \bar{z}- \bar{z}) \leq c \, h \lVert \bar{z} - \bar{z}_h \rVert_{L^2(\Omega)} \lVert  \bar{z} \rVert_{\mathbb{H}^{1}(\Omega)}.
\end{equation}
The estimate of the third term can be shown analoguous to \eqref{eq:full_2_1_terms} with an application of (L2)
and yields
\begin{align}
  (\bar p_h- \bar p,Q_h \bar z - \bar z) 
  &\leq c \, h \left( \lVert \bar{z} \rVert_{\mathbb{H}^{1}(\Omega)} + \lVert u_d \rVert_{\mathbb{H}^{\max{(0, 1-2s+\varepsilon)}}(\Omega)} \right)
  \lVert Q_h \bar z - \bar z \rVert_{L^2(\Omega)} \nonumber \\
  &\leq c \, h ^2\left( \lVert \bar{z} \rVert_{\mathbb{H}^{1}(\Omega)} + \lVert u_d \rVert_{\mathbb{H}^{\max{(0, 1-2s+\varepsilon)}}(\Omega)} \right)\lVert \bar z \rVert_{\mathbb{H}^{1}(\Omega)}
  \label{eq:full_3_2_terms}
\end{align}
Estimates \eqref{eq:full_2_1_terms} -- \eqref{eq:full_3_2_terms} together with an 
appropriate application of H\"older's and Young's inequality yield the desired estimate \eqref{eq:full_a_priori_z}.

In order to prove estimate \eqref{eq:full_a_priori_u} we proceed as follows. Introducing intermediate functions, applying the triangle inequality and using the stability results from Lemma \ref{lemma:stab} yields
\begin{equation}
\label{eq:full_a_priori_u_1}
\begin{aligned}
\norm{\bar{u} - \bar{u}_{h}}_{\mathbb{H}^{s}(\Omega)} &\leq \norm{(S-S_{h})\bar{z}}_{\mathbb{H}^{s}(\Omega)} + \norm{S_{h}(\bar{z} - \bar{z}_{h})}_{\mathbb{H}^{s}(\Omega)}\\
&\leq \norm{(S-S_{h})\bar{z}}_{\mathbb{H}^{s}(\Omega)} + c \norm{\bar{z} - \bar{z}_{h}}_{L^{2}(\Omega)}.
\end{aligned}
\end{equation}
Hence an application of the a priori error estimate \eqref{eq:fe_approximation_error} with $\delta = 1 - \varepsilon'$ and estimate \eqref{eq:full_a_priori_z} proves \eqref{eq:full_a_priori_u}.
The third estimate is obtained in the same way.
\end{proof}

\begin{remark}
  We see, that the rates, that are obtained by Theorem\nobreakspace \ref {thm:FD:rates} are not optimal with respect to the state
  and that they are dictated by the optimal linear rate, that we obtain for the control.
  In Section\nobreakspace \ref {sec:Numerics} we numerically measure higher rates for the optimal state, namely the same
  as for the variational discretization. The proof of higher rates will be addressed in future work and is mainly based on supercloseness results \cite{MeyerRoesch2004} for the control. These convergence rates carry over to the rates for the discrete control computed by the so-called post-processing step, i.e. using the projection formula \textup {(\ref {eq:disc_proj})} to obtain a new, piecewise linear approximation of the control. Numerical experiments for this post-processing approach are also contained in Section\nobreakspace \ref {sec:Numerics}. The theoretical analysis is left for future work.
\end{remark}

\section{Implementation}
\label{sec:Implementation}

In this section, we introduce a solver for the finite element approximation
\textup {(\ref {eq:dunford_taylor_FE})}.
The use of the Balakrishnan formula for inverting the fractional operator leads to the necessity of solving a large number of independent linear systems of equations to obtain an accurate solution. 
However, these systems carry a lot of structure that can be used to design efficient 
iterative schemes based on tailored Krylov subspace methods.

Following \cite{Bonito}, application of the sinc quadrature to the
Balakrishnan representation \eqref{eq:balakrishnan_formula}
gives rise to the discretization of the state
equation \eqref{eq:state_equation}.
For convenience, we repeat the resulting approximation here.
\begin{equation}
  \label{eq:im:uhk}
u_{h}^{k} = \frac{\sin{(s \pi)}}{\pi} k \sum_{l = -N_{-}}^{N_{+}} e^{(1-s)kl}v_{h}^{l}
\end{equation}
where 
$v_{h}^{l} \in \mathbb{U}(\mathcal{T}_{h})$ is the unique solution of the Galerkin variational problem
\begin{equation}
\label{eq:DT_FEM_systems}
\int_{\Omega} \nabla v_{h}^l \cdot \nabla w_{h} dx 
+ e^{kl} \int_{\Omega} v_{h}^l w_{h} dx  
= \int_{\Omega} z w_{h} dx \quad \forall w_{h} \in \mathbb{U}(\mathcal{T}_{h}).
\end{equation}
The evaluation of \textup {(\ref {eq:im:uhk})} requires the solution of  $N_{+} + N_{-} + 1$ linear
systems of the form
\begin{equation}
  \label{eq:im:DiscreteDTSystem}
\left( A + \alpha_l M \right) V^l = Z, \quad -N_{-} \leq l \leq N_{+}.
\end{equation}
Here, $\alpha_l = e^{kl}$ and $A$, $M$ denote respectively the corresponding stiffness and mass
matrices of the system, $Z$ denotes the load vector, while $V^l$ denotes the node vector for
$v^l_h$.
Notice that the $N_{-}~+~N_{+}~+~1$ linear systems in~\eqref{eq:im:DiscreteDTSystem} are independent 
for different values of~$l$.
Hence, a first approach for solving systems \textup {(\ref {eq:im:DiscreteDTSystem})} might be the use of massive
parallelization.
However, we shall follow a more efficient approach that exploits the structure of the
linear systems and uses tailored conjugated gradients solvers.
 
We start by normalizing the systems.
Application of a standard mass-lumping strategy results in a diagonal mass matrix~$M_h$. 
We  define $\rho := \|M_h^{-1/2} A  M_h^{-1/2}\|_{\infty}$,
$\tilde{A} =  \frac{1}{\rho}{M_h}^{-1/2}A{M_h}^{-1/2}$, 
$\tilde{\alpha}_l = \frac{1}{\rho}\alpha_l$, 
$\tilde{V}^l = {M_h}^{1/2} V^l$ 
and $\tilde{Z} = \frac{1}{\rho}{M_h}^{-1/2}Z$.
Then, the linear systems~\textup {(\ref {eq:im:DiscreteDTSystem})} can be reformulated as
\begin{equation}
  \label{eq:im:DiscreteDTSystem_modified}
\Big( \tilde{A} + \tilde{\alpha}_l I \Big) \tilde{V}^l = \tilde{Z}, \quad -N_{-} \leq l \leq N_{+}.
\end{equation}
We can estimate the $2-$condition number of system $l$ in \textup {(\ref {eq:im:DiscreteDTSystem_modified})}
by
\begin{align}
  \label{eq:im:estimateCond}
  \kappa\left(\tilde{A} + \tilde{\alpha}_l I\right) = 
  \frac{\lambda_{\max}(\tilde{A} + \tilde{\alpha}_l I)}
  {\lambda_{\min}(\tilde{A} + \tilde{\alpha}_l I)}
  = 
    \frac{\lambda_{\max}(\tilde{A}) + \tilde{\alpha}_l}
  {\lambda_{\min}(\tilde{A}) + \tilde{\alpha}_l} 
  \leq 1 + \min\left(
\frac{\lambda_{\max}(\tilde{A})}{\tilde{\alpha}_l},
 \kappa(\tilde{A})
\right)
\end{align}
where $\lambda_{\max}(\tilde A)$ and $\lambda_{\min}(\tilde A)$ denote the largest and smallest
eigenvalue of the symmetric positive definit matrix $\tilde A$, respectively.
From \textup {(\ref {eq:im:estimateCond})} we observe, that for small $\tilde \alpha_l$ the condition number of
$\tilde A + \tilde \alpha_l I$ is close to the condition number of $\tilde A$, which is a scaled
stiffness matrix, while for large $\tilde \alpha_l$ the condition number converges to 1.
By introducing the scaling with $\rho$, we fix $\lambda_{\max}(\tilde A) \leq 1$.

Thus for $l$ decreasing from $N_+$ to $N_-$ the condition number of the linear system $\tilde A + \tilde \alpha_l I$
is increasing. While for $l \equiv N_+$ conjugated gradients without preconditioning is a well suited solver, 
for $l\equiv N_-$ preconditioning in general is required. 
Due to this observation, we consider two adapted linear solvers.

\begin{itemize}
\item Linear problems, for which $l$ is sufficiently large, 
are considered to be well-conditioned, and no further preconditioning is needed to obtain fast convergence of the conjugate gradient solver.
Thanks to  the shift-invariance property of Krylov subspace methods 
the Krylov spaces that are generated during the conjugate gradients method are independent of $l$.
As the build-up of the Krylov space contains the only matrix-vector multiplication in the conjugated gradients method,
the dimension of the space is equal to the number of matrix-vector multiplications. We fix a number $N_{\max}$ of multiplications and proceed as follows.
Starting with $l\equiv N_+$ we solve linear systems for decreasing $l$, where we reuse the Krylov spaces from previous solutions.
We stop at $l\equiv N_0$ as soon as the required Krylov space has reached the dimension of $N_{\max}$. 

For the implementation, we use a variant of 
the conjugate gradient method proposed in~\cite{Frommer1999}.
In Algorithm\nobreakspace \ref {alg:im:ShiftInvariant} we summarize the
pseudo code.
\item For the resulting systems $N_-,\ldots, N_0$ preconditioning is necessary which conflicts
with the shift invariant property of Krylov methods.
Note that subsequent systems are still similar, and thus, that a preconditioner for system
$l$ is also a (worse, but not necessarily bad)  preconditioner for system $l+1$. Therefore, we use the standard
approach for solving system \textup {(\ref {eq:im:DiscreteDTSystem_modified})} sequentially and recalculating a
new preconditioner whenever the old one is no longer good enough, i.e. as soon as a given maximum number of iterations is exceeded in a conjugated gradients method, see Algorithm\nobreakspace \ref {alg:im:sequential}.

\end{itemize}

In Section\nobreakspace \ref {ssec:num:DunfordTaylor} we report on the behaviour of the proposed solver.

\begin{algorithm}
\SetAlgoLined
\KwIn{$\tilde A,\tilde \alpha,\tilde Z, N_{\max}$}
\KwData{Set: $\mathcal K=\emptyset$, $l=N_+$}
\KwOut{$N_0$}
\While{$dim (\mathcal K) < N_{\max}$}
{ 
\For(\tcp*[h]{cg-iteration for system $l$}){$k=1,\ldots$}
{
\If{$k>\mbox{dim}(\mathcal K)$}
{ 
Calculate basis vector for $k$-th Krylov space and store into $\mathcal K$\;
\label{alg:im:ShiftCG_increaseKrylovSpace}
}
{}
Solve linear system $l$ in space $\mathcal K_k$
using \cite[Alg.~4]{Frommer1999}\;
\label{alg:im:ShiftCG_solveInKrylovSpace}
}
$l:=l-1$\;
}
$N_0 := l$
\caption{Pseudo code for solving the well-conditioned systems. The only matrix-vector
multiplication appears in line\nobreakspace \ref {alg:im:ShiftCG_increaseKrylovSpace}.
Here the space $\mathcal K_k$ in line\nobreakspace \ref {alg:im:ShiftCG_solveInKrylovSpace} is
the span over the first $k$ basis elements of $\mathcal K$.}
\label{alg:im:ShiftInvariant}
\end{algorithm} 

\begin{algorithm}
\SetAlgoLined
\KwIn{$\tilde A,\tilde \alpha,\tilde Z, N_0$}
\KwData{Set: $N_{\max}>0$, $N_{iter} = N_{\max}+1$, }
\For{$l=-N_-\ldots N_0$}
{ 
\If{$N_{iter} >N_{\max}$ }
{
Build up amg preconditioner  $P$ for system $l$\;
}
Solve systems $l$ with preconditioned conjugate gradients method using preconditioner $P$ with
$N_{iter}$ iterations\;
}
\caption{Pseudo code for solving the not well-conditioned systems.}
\label{alg:im:sequential}
\end{algorithm}

\begin{remark}[An alternative solver]
  In \cite{Chan1999} a conjugate gradients method is proposed that uses Krylov spaces generated for one of the linear systems \textup {(\ref {eq:im:DiscreteDTSystem_modified})}, called seed
  system, to generate good initial values, or even solutions, for the other systems.
  Thanks to the particular structure of
  the systems \textup {(\ref {eq:im:DiscreteDTSystem_modified})} this can be done without additional
  matrix-vector multiplications, see \cite[Sec.~3.1]{Chan1999}. 
  Upon convergence, another system is chosen as a seed system, for which the Krylov spaces are
  generated.
  In our implementation, we combine this approach with algebraic multigrid (amg) preconditioning, 
  and choose that system as the next seed system, that
  currently has the largest residuum.
  
  This approach requires storing the solution to all systems in memory to apply the Krylov spaces and to find the next seed system. Unfortunately, this turned out to be not feasible for
  fine meshes in 3D, but we obtained very fast convergence of the method, when applicable.
  
  A combination of the proposed sequential solver as in Algorithm\nobreakspace \ref {alg:im:sequential} and the solver
  proposed in \cite{Chan1999} seems possible (at least with some restrictions) and will be subject
  to future work.
\end{remark}
Finally we note that a large number of tailored Krylov methods is proposed to deal with shifted
systems that require preconditioning and  we only refer to
\cite{2003_Benzi_ApproximateInversePrecond_for_Shift,
2003_Frommer_bicgstab_shiftedSystems,
2016_Soodhalter_recursiveGMRES_generalPrecond,
2014_SoodhalterSzydXue_KrylovRecycling_shiftesSystems,
2017_ZhongGu_FlexibleAdaptiveSGMRES_shiftedSystems}.

\section{Numerical results}
\label{sec:Numerics}

In this section, we validate the theoretical rates of convergence derived in Section\nobreakspace \ref {sec:error_estimates}. 
In Section\nobreakspace \ref {ssec:num:DunfordTaylor}, we investigate the solver 
 for the fractional Laplace proposed in Section\nobreakspace \ref {sec:Implementation}.
In Section\nobreakspace \ref {ssec:OptControl}, we present the convergence rates of the fully discrete finite element scheme for the approximation of the optimal control problem.  
All numerical experiments are conducted for a range of values of the fractional exponent~$s$.

We implement the solver proposed in Section\nobreakspace \ref {sec:Implementation} in C++ using the
PETSc linear algebra package \cite{petsc-web-page} and solve the optimal control problem
using the TAO package of PETSc using the bound-constrained limited-memory variable-metric method
(tao\_blmvm), which is a limited memory BFGS method.
 We generate meshes, finite element functions and assemble matrices
using FEniCS \cite{LoggMardalEtAl2012a} through the C++ interface.

 For the solver of the fractional operator proposed in Section\nobreakspace \ref {sec:Implementation} we fix $N_{\max} = 500$ for 2D simulations and 
 $N_{\max} = 250$ for 3D simulations. 
 The individual linear systems are solved up to a relative accuracy of $10^{-8}$. 
 For Algorithm\nobreakspace \ref {alg:im:sequential}, we calculate a  
 new preconditioner as soon as more than $N_{\max}=20$ iterations are taken in the preconditioned conjugate gradients method.
 As amg preconditioner we use  2 V-cycles of Hypre \cite{hypre-web-page} that is accessed through the PETSc interface. We stop the optimization as soon as the $l^2$-norm of the projected gradient is smaller or equal to $10^{-5} \sqrt{h^n}$, where $h$ is the length of the longest edge in the finite element mesh. Here the scaling with $h$ mimics the different scaling of the $l^2$-norm and the
 $L^2(\Omega)$-norm.

\subsection{The solver for the fractional operator}
\label{ssec:num:DunfordTaylor}
Let us first report on the performance of the proposed solver for the systems
\textup {(\ref {eq:im:DiscreteDTSystem_modified})}.
As a test example we use $\Omega = (0,1)^n$, $n \in \{2,3\}$ and
set $f = \min(0.25,f_0)$, where $f_0(m) = 0.5$, with $m$ the center of $\Omega$,  
$f_0|_{\partial\Omega} = 0$ and $f_0$ is linearly interpolated between these values.
Note that no analytical solution is known for this
right-hand side $f$, and that $f$ enjoys $\mathbb{H}^{3/2-\varepsilon}(\Omega)$ regularity, which is the maximal regularity of the optimal control $z$.
We solve the equation $(-\Delta)^s u_h^k = f$ on a sequence of homogeneously refined meshes and
use the solution on the finest mesh (with $N_\Omega = 4198401$ nodes for $n=2$) as the reference solution. 
These meshes are chosen, such that all
kinks in $f$ are resolved, and the integration of $f$ is done with no numerical error.

In Table\nobreakspace \ref {tab:num:DunfordTaylor_2D} and Table\nobreakspace \ref {tab:num:DunfordTaylor_3D} we report on the solver for the cases $n=2$ and $n=3$, respectively, and for $s=0.05$ and $s=0.5$.
 
\begin{table}
\centering
\begin{tabular}{rcccc}
$N_\Omega$ & $N_{\alpha} (s = 0.05)$ & $N_{\alpha} (s= 0.5)$ &  $N_{\alpha} (s= 0.95)$ &$\#amg\,\,setup$\\ 
\hline 
     25 &   58 (58/0)     &   13 (13/0)    &   58 (58/0)    & 0 \\ 
     81 &  158 (158/0)    &   31 (31/0)    &  158 (158/0)   & 0 \\ 
    289 &  308 (308/0)    &   61 (61/0)    &  308 (308/0)   & 0 \\ 
   1089 &  508 (508/0)    &   99 (99/0)    &  508 (508/0)   & 0 \\ 
   4225 &  757 (757/0)    &  145 (145/0)   &  757 (757/0)   & 0 \\ 
  16641 & 1056 (1056/0)   &  203 (203/0)   & 1056 (1056/0)  & 0 \\ 
  66049 & 1406 (1406/0)   &  269 (269/0)   & 1406 (1406/0)  & 0 \\ 
 263169 & 1806 (1677/129) &  345 (135/210) & 1806 (54/1752) & 1 \\ 
1050625 & 2254 (2089/165) &  429 (163/266) & 2254 (62/2192) & 2 \\ 
4198401 & 2753 (2548/205) &  525 (196/329) & 2753 (72/2681) & 2 
\end{tabular}

\caption{The behavior of the proposed solver for the fractional operator for 2D simulation.
$N_\Omega$ denotes the number of degrees of freedom in $\Omega$, $N_\alpha = N_- + N_+ + 1$ denotes the number of linear systems to solve. In brackets, we show how many systems are solved by
Algorithm\nobreakspace \ref {alg:im:ShiftInvariant} and  Algorithm\nobreakspace \ref {alg:im:sequential} respectively.
We give results for $s=0.05$, $s=0.5$, and $s=0.95$.
The number of amg setups in Algorithm\nobreakspace \ref {alg:im:sequential} are equal in all cases.
}
\label{tab:num:DunfordTaylor_2D}
\end{table}
 
 \begin{table}
 \centering
\begin{tabular}{rcccc}
$N_\Omega$ & $N_{\alpha} (s = 0.05)$ & $N_{\alpha} (s= 0.5)$& $N_{\alpha} (s= 0.95)$ & $\#amg\,\,setup$\\ 
\hline 
    125 &   38 (38/0)  &    9 (9/0)   &  38 ( 38/0)  & 0 \\ 
    729 &  124 (124/0) &   25 (25/0)  & 124 (124/0)  & 0 \\ 
   4913 &  258 (258/0) &   51 (51/0)  & 258 (258/0)  & 0 \\ 
  35937 &  444 (444/0) &   85 (85/0)  & 444 (444/0)  & 0 \\ 
 274625 &  678 (678/0) &  131 (131/0) & 678 (678/0)  & 0 \\ 
2146689 &  964 (895/69)&  185 (73/112)& 964 (30/934)  & 1
\end{tabular}

\caption{Behavior of the proposed solver for the fractional operator for 3D simulation.
For an explanation of the abbreviations, see Table\nobreakspace \ref {tab:num:DunfordTaylor_2D}.
}
\label{tab:num:DunfordTaylor_3D}
\end{table} 

We observe that in fact, the number of amg setups is very small or a set up is not even necessary, which indicates, how closely
related the systems are.

Finally, for small $s$ the operator is closer to identity, and thus, more systems are
well-conditioned, which can be seen by the number of systems that are solved by
Algorithm\nobreakspace \ref {alg:im:ShiftInvariant} in comparison to the number of systems solved by
Algorithm\nobreakspace \ref {alg:im:sequential}.

Let us briefly comment on the convergence rate from Corollary\nobreakspace \ref {cor:fe_error_estimates_regular} for $u_h^k$.
As the above defined right hand side enjoys $f \in \mathbb{H}^{3/2-\varepsilon}(\Omega)$,
we expect a rate of $h^{\min(2,3/2+2s-\varepsilon')}$.
In Table\nobreakspace \ref {tab:num:DF_ratesL2} we show the observed convergence rates for $n=2$ and $s\in \{0.05,0.1,0.25\}$, 
which indeed confirm the theoretical predictions.  For $n=3$ memory consumption
restricts the quality of the reference solution, such that we do not measure a rate for $n=3$.
  
\begin{table}
 
 \begin{tabular}{rccccccc}
& & \multicolumn{2}{c}{(s=0.05)} & \multicolumn{2}{c}{(s=0.10)} & \multicolumn{2}{c}{(s=0.25)}\\ 
$N_\Omega$ & $h$ & $\eta_{L^2(\Omega)}$ & $r^{0.05}_{L^2(\Omega)}$ & $\eta_{L^2(\Omega)} $  & $ r^{0.10}_{L^2(\Omega)}$ & $\eta_{L^2(\Omega)}$ & $r^{0.25}_{L^2(\Omega)}$ \\ 
\hline 
     25 & 0.3536 & 0.038866 & 0.00 & 0.033601 & 0.00 & 0.009917 & 0.00\\ 
     81 & 0.1768 & 0.012021 & 1.69 & 0.009930 & 1.76 & 0.002924 & 1.76\\ 
    289 & 0.0884 & 0.003728 & 1.69 & 0.002913 & 1.77 & 0.000662 & 2.14\\ 
   1089 & 0.0442 & 0.001190 & 1.65 & 0.000871 & 1.74 & 0.000158 & 2.07\\ 
   4225 & 0.0221 & 0.000387 & 1.62 & 0.000265 & 1.72 & 0.000038 & 2.05\\ 
  16641 & 0.0110 & 0.000127 & 1.61 & 0.000081 & 1.71 & 0.000009 & 2.09\\ 
  66049 & 0.0055 & 0.000041 & 1.62 & 0.000025 & 1.72 & 0.000002 & 2.06\\ 
 263169 & 0.0028 & 0.000013 & 1.68 & 0.000007 & 1.77 & 0.000001 & 2.10
\end{tabular}

  \caption{Experimental convergence rates for the numerical solution of the fractional Laplace
  equation with $f \in \mathbb{H}^{3/2-\varepsilon}(\Omega)$. 
  We observe, that the expected rate $r^s_{L^2(\Omega)} =  \min(2,3/2+2s-\varepsilon')$ is fulfilled for all examples.
  } 
  \label{tab:num:DF_ratesL2}    
\end{table}

\subsection{Optimal Control Problem}\label{ssec:OptControl}
To verify the theoretical convergence rates of the finite element discretization of the optimal control problem, 
we perform numerical experiments without a known optimal solution.
We use the domain $\Omega = (0, 1)^2$ and set the desired state to be equal to an eigenfunction of the Laplacian on the square, 
namely $u_d = \sin (2 \pi x) \sin (2 \pi y)$ and $f \equiv 0$. 
We consider three different values of the fractional parameter, 
namely $s \in \lbrace 0.05, 0.25, 0.5\rbrace$ and choose $a = -0.8, b = 0.8$, 
such that the box-constraints are attained in some subdomain of~$\Omega$.
The optimal solution for $h=0.0014$ is considered as reference solution.

Results of the numerical tests are summarized in Figure\nobreakspace \ref {fig:OPT_Plots}. 
First order convergence of the approximation of the control is obtained, 
which is in line with~\textup {(\ref {eq:full_a_priori_z})}. 

We also report on results using the post-processing approach \cite{MeyerRoesch2004}.
Here the projection formula \textup {(\ref {eq:disc_proj})} is used to obtain a higher order approximation for the optimal control.
Higher order means, that instead of an approximation with piecewise constant functions, 
a piecewise linear approximation is obtained that has the same structure as the optimal control obtained with variational discretization. 
We expect thus the same optimal rate of convergence for this post-processed optimal control $\bar{z}_h^{PP}$ 
as for the variational discretization approach, namely 
$\|\bar{z}-\bar{z}_h^{PP}\|_{L^2(\Omega)} \leq c \, h^{\min(2,3/2+2s-\varepsilon')}$.
In Figure\nobreakspace \ref {fig:OPT_Plots} we observe the expected higher rates for the optimal controls with post-processing approach for $n=2$.

We also investigate the finite element approximation of the state in the $L^2(\Omega)$- and $H^s(\Omega)$-norms, 
the latter being estimated using Gagliardo-Nirenberg interpolation inequality 
$\| \bar{u} - \bar{u}_h\|_{H^s(\Omega)} \lesssim \| \bar{u} - \bar{u}_h\|^{1 - s}_{L^2(\Omega)} \|\bar{u} - \bar{u}_h \|^s_{H^1(\Omega)}$. We observe $h^{\min(2,3/2 + 2s)}$ order of convergence in the $L^2(\Omega)$-norm and $h^{\min(2-s,3/2 + s)}$ order of convergence in the $H^s(\Omega)$-norm. 
The theoretical justification as well as the analysis of the post-processing approach is left for future work.
\begin{figure}[!]
  

%
%
\begin{tikzpicture}

\begin{axis}[%
width=6cm,
height=6.5cm,
at={(0in,0in)},
scale only axis,
unbounded coords=jump,
xmode=log,
xmin=0.001,
xmax=1,
xminorticks=true,
xlabel style={},
xlabel={Mesh size $h$},
ymode=log,
ymin=1e-07,
ymax=1,
yminorticks=true,
ylabel style={},
ylabel={},
axis background/.style={fill=white},
title style={},
title={$\lVert \overline z - \overline z_h\rVert_{L^2(\Omega)}$},
legend style={at={(0.97,0.03)}, anchor=south east, legend cell align=left, align=left, draw=white!15!black}
]

\addlegendimage{only marks, mark = x, mark options={solid, black,scale=2}}
\addlegendentry{$s=0.05$}

\addlegendimage{only marks, mark = o, mark options={solid, black,scale=2}}
\addlegendentry{$s=0.25$}

\addlegendimage{only marks, mark = asterisk, mark options={solid, black,scale=2}}
\addlegendentry{$s=0.50$}

\addlegendimage{no markers,style =  solid}
\addlegendentry{no post proc.} 

\addlegendimage{no markers,style =  dotted}
\addlegendentry{post proc.}

\addlegendimage{line width=2pt}
\addlegendentry{$\mathcal O(h)$, $\mathcal O(h^2)$}

\addplot [color=black, mark=x, mark options={solid, black,scale=2},line width=1pt]
  table[row sep=crcr]{%
0.3536	0.26939\\
0.1768	0.13574\\
0.0884	0.056752\\
0.0442	0.028257\\
0.0221	0.013718\\
0.011	0.0066289\\
0.0055	0.0031792\\
0.0028	0.0014045\\
0.0014	nan\\
};

\addplot [color=black, dotted, mark=x, mark options={solid, black,scale=2},line width=1pt]
  table[row sep=crcr]{%
0.3536	0.090477\\
0.1768	0.093431\\
0.0884	0.031245\\
0.0442	0.011918\\
0.0221	0.0042411\\
0.011	0.001434\\
0.0055	0.00049232\\
0.0028	0.00019508\\
0.0014	nan\\
};

\addplot [color=black, mark=o, mark options={solid, black,scale=2},line width=1pt]
  table[row sep=crcr]{%
0.3536	0.13442\\
0.1768	0.066736\\
0.0884	0.033254\\
0.0442	0.016722\\
0.0221	0.0083596\\
0.011	0.0041603\\
0.0055	0.0020329\\
0.0028	0.00092431\\
0.0014	nan\\
};

\addplot [color=black, dotted, mark=o, mark options={solid, black,scale=2},line width=1pt]
  table[row sep=crcr]{%
0.3536	0.0666\\
0.1768	0.024359\\
0.0884	0.0049781\\
0.0442	0.0012801\\
0.0221	0.00033096\\
0.011	8.6548e-05\\
0.0055	2.7456e-05\\
0.0028	8.4754e-06\\
0.0014	nan\\
};

\addplot [color=black, mark=asterisk, mark options={solid, black,scale=2},line width=1pt]
  table[row sep=crcr]{%
0.3536	0.021818\\
0.1768	0.012591\\
0.0884	0.0059975\\
0.0442	0.0029942\\
0.0221	0.0014971\\
0.011	0.00074483\\
0.0055	0.00034179\\
nan	nan\\
nan	nan\\
};

\addplot [color=black, dotted, mark=asterisk, mark options={solid, black,scale=2},line width=1pt]
  table[row sep=crcr]{%
0.3536	0.015637\\
0.1768	0.0070358\\
0.0884	0.0016028\\
0.0442	0.00039882\\
0.0221	9.9398e-05\\
0.011	2.3872e-05\\
0.0055	5.8388e-06\\
nan	nan\\
nan	nan\\
};

\addplot [color=black,line width=2pt]
  table[row sep=crcr]{%
0.3536	0.7072\\
0.1768	0.3536\\
0.0884	0.1768\\
0.0442	0.0884\\ 
0.0221	0.0442\\
0.011	0.022\\
0.0055	0.011\\
0.0028	0.0056\\
0.0014	0.0028\\
};

\addplot [color=black, forget plot,line width=2pt]
  table[row sep=crcr]{%
0.3536	0.012503296\\
0.1768	0.003125824\\
0.0884	0.000781456\\
0.0442	0.000195364\\
0.0221	4.8841e-05\\
0.011	1.21e-05\\
0.0055	3.025e-06\\
0.0028	7.84e-07\\
0.0014	1.96e-07\\
};
\end{axis}
\end{tikzpicture}%
%
%
\begin{tikzpicture}

\begin{axis}[%
width=6cm,
height=6.5cm,
at={(0in,0in)},
scale only axis,
unbounded coords=jump,
xmode=log,
xmin=0.001,
xmax=1,
xminorticks=true,
xlabel style={},
xlabel={Mesh size $h$},
ymode=log,
ymin=1e-08,
ymax=1,
yminorticks=true,
ylabel style={},
ylabel={},
axis background/.style={fill=white},
title style={},
title={$\lVert \overline u - \overline u_h\rVert_{L^2(\Omega)}$ and $\lVert \overline u - \overline u_h\rVert_{H^s(\Omega)}$},
legend style={at={(0.97,0.03)}, anchor=south east, legend cell align=left, align=left, draw=white!15!black}
]

\addlegendimage{only marks, mark = x, mark options={solid, black,scale=2}}
\addlegendentry{$s=0.05$}

\addlegendimage{only marks, mark = o, mark options={solid, black,scale=2}}
\addlegendentry{$s=0.25$}

\addlegendimage{only marks, mark = asterisk, mark options={solid, black,scale=2}}
\addlegendentry{$s=0.50$}

\addlegendimage{no markers,style =  solid}
\addlegendentry{$L^2$-error}

\addlegendimage{no markers,style =  dotted}
\addlegendentry{$H^s$-error}

\addlegendimage{line width=2pt}
\addlegendentry{$\mathcal O(h^{\frac{3}{2}})$, $\mathcal O(h^2)$}

\addplot [color=black, mark=x, mark options={solid, black,scale=2},line width=1pt]
  table[row sep=crcr]{%
0.3536	0.13857\\
0.1768	0.05708\\
0.0884	0.016835\\
0.0442	0.004415\\
0.0221	0.0013773\\
0.011	0.00046746\\
0.0055	0.00016525\\
0.0028	5.7558e-05\\
0.0014	nan\\
};

\addplot [color=black, dotted, mark=x, mark options={solid, black,scale=2},line width=1pt]
  table[row sep=crcr]{%
0.3536	0.15774\\
0.1768	0.066956\\
0.0884	0.020611\\
0.0442	0.0056119\\
0.0221	0.0018191\\
0.011	0.0006384\\
0.0055	0.00023332\\
0.0028	8.3441e-05\\
0.0014	nan\\
};

\addplot [color=black, mark=o, mark options={solid, black,scale=2},line width=1pt]
  table[row sep=crcr]{%
0.3536	0.058655\\
0.1768	0.023263\\
0.0884	0.0060181\\
0.0442	0.0014904\\
0.0221	0.00040044\\
0.011	0.00010283\\
0.0055	3.0849e-05\\
0.0028	7.9687e-06\\
0.0014	nan\\
};

\addplot [color=black, dotted, mark=o, mark options={solid, black,scale=2},line width=1pt]
  table[row sep=crcr]{%
0.3536	0.10305\\
0.1768	0.044962\\
0.0884	0.013785\\
0.0442	0.0041221\\
0.0221	0.0012969\\
0.011	0.00039839\\
0.0055	0.00013496\\
0.0028	4.082e-05\\
0.0014	nan\\
};

\addplot [color=black, mark=asterisk, mark options={solid, black,scale=2},line width=1pt]
  table[row sep=crcr]{%
0.3536	0.0031609\\
0.1768	0.0014215\\
0.0884	0.00036551\\
0.0442	8.7986e-05\\
0.0221	2.2408e-05\\
0.011	6.1086e-06\\
0.0055	1.1771e-06\\
nan	nan\\
nan	nan\\
};

\addplot [color=black, dotted, mark=asterisk, mark options={solid, black,scale=2},line width=1pt]
  table[row sep=crcr]{%
0.3536	0.0097581\\
0.1768	0.0051452\\
0.0884	0.0017609\\
0.0442	0.00060367\\
0.0221	0.00021451\\
0.011	7.884e-05\\
0.0055	2.4219e-05\\
nan	nan\\
nan	nan\\
};

\addplot [color=black,line width=2pt]
  table[row sep=crcr]{%
0.3536	0.841062705448292\\
0.1768	0.297360571212795\\
0.0884	0.105132838181036\\
0.0442	0.0371700714015994\\
0.0221	0.0131416047726296\\
0.011	0.00461475893194867\\
0.0055	0.00163156366716105\\
0.0028	0.000592648293678468\\
0.0014	0.000209532813659341\\
};

\addplot [color=black, forget plot,line width=2pt]
  table[row sep=crcr]{%
0.3536	0.0025006592\\
0.1768	0.0006251648\\
0.0884	0.0001562912\\
0.0442	3.90728e-05\\
0.0221	9.7682e-06\\
0.011	2.42e-06\\
0.0055	6.05e-07\\
0.0028	1.568e-07\\
0.0014	3.92e-08\\
};
\end{axis}
\end{tikzpicture}%

\caption{Convergence rates of the discretization of the optimal control problem in 2D. 
Figure on the left-hand side presents approximations of the control~$z$. 
First order of convergence of the piecewise constant approximation can be observed. 
Application of the additional post-processing significantly improves the convergence properties, 
and we observe $h^{\min(2,3/2 + 2s)}$ convergence. 
On the right-hand side convergence of the approximation of the state~$u$ is shown. 
Convergence order of the piecewise linear finite element method measured in the $L^2(\Omega)$-norm depends 
on the choice of~$s$ and varies between $3/2$ and $2$, which can be attained for sufficiently large~$s$. 
Convergence order in $H^s(\Omega)$-norm is included for completeness.}
\label{fig:OPT_Plots}
\end{figure}
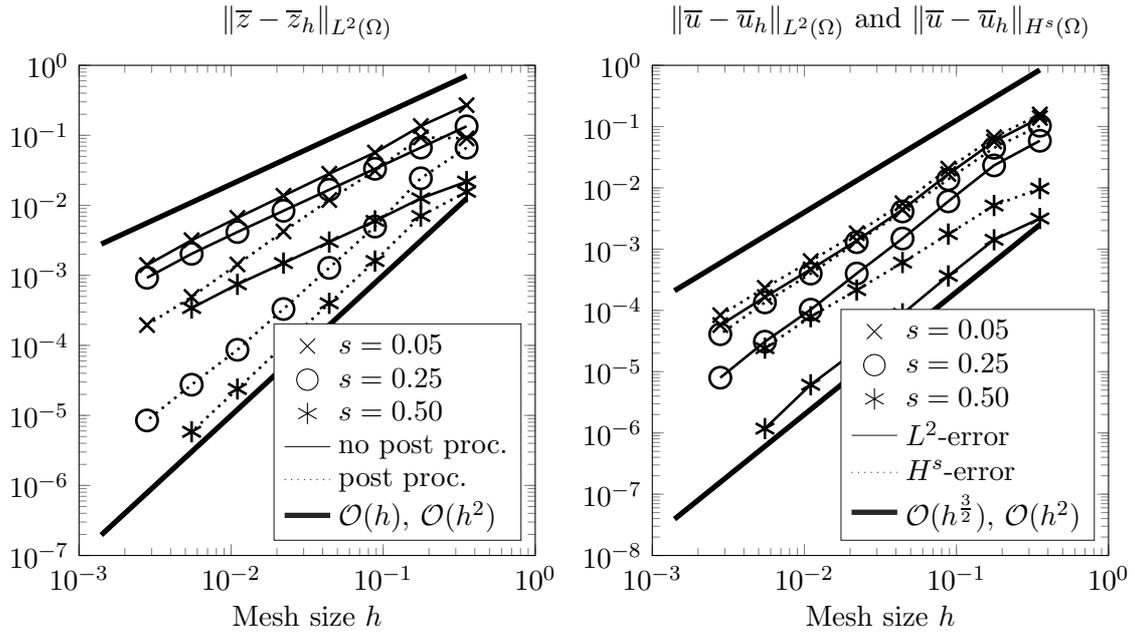
 
\bigskip

\begin{acknowledgement}
 We would like to thank Harbir Antil for giving a compact course on fractional PDEs
 as a preparation for this work.
 Furthermore, we would like to thank Johannes Pfefferer for many fruitful discussions during the
 preparation of this work.
\end{acknowledgement}

\bibliographystyle{plain}

\end{document}